\documentclass[12pt, a4paper, draft]{article}
\usepackage{amsmath,amssymb,amsthm,amsfonts}
\usepackage[numbers,sort&compress]{natbib}

\newcommand\R{\mathbb{R}}

\newcommand\N{\mathbb{N}}

\newcommand\C{\mathbb{C}}
\newcommand\No{\mathbb{N}_0}

\newcommand\Rp{\mathbb{R}^{+}}
\newcommand\re{\mathrm{e}}

\newcommand\abs[1]{{\lvert#1\rvert}}

\newcommand\norm[1]{{\lVert #1\rVert}}
\newcommand{\dm}{\mathop{}\!\mathrm{d}}
\newcommand\ob[3][\relax]{\mathrm{B}%
\ifx\relax#1\relax\else%
_{#1}\fi(#2,#3)}
\newcommand\cb[3][\relax]{\bar{\mathrm{B}}%
\ifx\relax#1\relax\else%
_{#1}\fi(#2,#3)}

\newcommand\emptynorm{\lVert\,\cdot\,\rVert}
\newcommand\emptyabs{\lvert\,\cdot\,\rvert}

\newcommand\adj{^*} 
\newcommand\pin{^-} 
\newcommand\pfi{^+} 
\newcommand{\pl}[1]{\Lambda(#1)} 

\newcommand\F{\mathcal{F}}

\newcommand\od[2]{D^{\ifx\relax#1\relax\else(#1)\fi}(#2)} 
\newcommand\fd[3][{\F}]{D_{#1}^{\ifx\relax#2\relax\else(#2)\fi}(#3)} 
\newcommand\plp{^{\mathrm{pl}}}
\newcommand\pln{^{\mathrm{no}}}

\newcommand\diff[1]{^{(#1)}} 
\newcommand\fdiff[1]{^{[#1]}} 

\newcommand\cm{^{\mathrm{c}}}

\newcommand\clstar[2]{\ast^{#1}_{#2}}
\newcommand\refcl[2]{($\clstar{#1}{#2}$)}

\DeclareMathOperator\supp{supp}
\DeclareMathOperator\tint{int}

\newtheorem{theorem}{Theorem}[section]

\newtheorem{proposition}[theorem]{Proposition}

\newtheorem{corollary}[theorem]{Corollary}
\newtheorem{lemma}[theorem]{Lemma}

\theoremstyle{definition}
\newtheorem{definition}[theorem]{Definition}

\newtheorem{question}[theorem]{Question}

\theoremstyle{remark}

\newcommand\theaddress{}
\newcommand\address[1]{\renewcommand\theaddress{#1}}
\newcommand\theemails{}
\makeatletter
\newcommand\email[1]{%
  \ifx\relax\theemails\relax
    \renewcommand\theemails{Email address: {\ttfamily#1}}
  \else
    \g@addto@macro\theemails{\\Email address: {\ttfamily#1}}
  \fi}
\makeatother
\AtEndDocument{\vskip 2em\noindent\theaddress\null\smallskip\theemails}

\newcommand\keywords[1]{%
  \let\qbfatempa\thefootnote
  \renewcommand\thefootnote{}
  \footnote{\textit{Key words and phrases}. #1}
  \let\thefootnote\qbfatempa
}
\newcommand\subjclass[2][1990]{%
  \let\qbfatempa\thefootnote
  \renewcommand\thefootnote{}
  \footnote{#1 \textit{Mathematics Subject Classification}. #2}
  \let\thefootnote\qbfatempa
}

\title{Quasianalyticity in certain Banach function algebras}
\author{J. F. Feinstein \and S. Morley\thanks{This paper contains work from the second author's PhD thesis.}\hphantom{i}\thanks{The second author was supported by EPSRC grant number {EP/L50502X/1}}}
\address{School of Mathematical Sciences, The University of Nottingham, University Park, Nottingham, NG7 2RD, UK}
\email{joel.feinstein@nottingham.ac.uk}
\email{pmxsm9@nottingham.ac.uk}

\date{}
\begin{document}
\maketitle


\begin{abstract}
Let $X$ be a perfect, compact subset of the complex plane. We consider algebras of those functions on $X$ which satisfy a generalised notion of differentiability, which we call $\mathcal{F}$-differentiability. In particular, we investigate a notion of quasianalyticity under this new notion of differentiability and provide some sufficient conditions for certain algebras to be quasianalytic. We give an application of our results in which we construct an essential, natural uniform algebra $A$ on a locally connected, compact Hausdorff space $X$ such that $A$ admits no non-trivial Jensen measures yet is not regular. This construction improves an example of the first author (2001).
\end{abstract}

Let $X$ be a perfect, compact subset of the complex plane $\C$. We consider those normed algebras consisting of complex-valued, continuously complex-differentiable functions on $X$, denoted $\od 1X$. These algebras were introduced by Dales and Davie in \cite{dales1973} and further investigated, for example, in \cite{bland2005} and \cite{dalefein2010}. The algebra $\od1X$ need not be complete, and the completion of $\od1X$ need not be a Banach function algebra in general.

Bland and the first author \cite{bland2005} introduced $\F$-differentiation, which generalises the usual complex-differentiation, and considered normed algebras of $\F$-differentiable functions, denoted $\fd1X$. The algebra $\fd1X$ is complete and $\od1X\subseteq\fd1X$.

Dales and Davie (\cite{dales1973}) also considered those algebras of complex-valued functions which have continuous complex-derivatives of all orders, and introduced the Dales-Davie algebras $\od{}{X,M}$. They defined a notion of quasianalyticity for these algebras and gave sufficient conditions for the algebra $\od{}{X,M}$ to be quasianalytic. (For the classical definition of quasianalytic collections of functions, see \cite[Chapter~19]{rudin1987}.)

In this paper, we define a notion of quasianalyticity for infinitely $\F$-differentiable functions (defined later) and give a sufficient condition for this new notion of $\F$-quasianalyticity. In certain cases, this sufficient condition will improve that given by Dales and Davie in \cite{dales1973}. We conclude the paper with the construction of a uniform algebra $A$ on a locally connected, compact Hausdorff space $X$ such that $A$ is essential and $A$ does not admit any non-trivial Jensen measures yet is not regular. (The relevant definitions are given in Section \ref{construction section}.) This construction improves an example of the first author from \cite{feinstein2001}.

\section{Definitions and Basic results}
Throughout this paper we say {\em compact plane set} to mean a non-empty, compact subset of the complex plane $\C$. We denote the set of non-negative integers by $\No$ and the set of positive integers by $\N$. Let $X$ be a compact Hausdorff space. We denote the algebra (with pointwise operations) of all continuous, complex-valued functions on $X$ by $C(X)$. For $E\subseteq X$, we set
\[
\abs{f}_E:=\sup_{x\in E}\abs{f(x)}\qquad(f\in C(X)).
\]
With the norm $\emptyabs_X$, $C(X)$ is a commutative, unital Banach algebra. Let $S$ be a subset of $C(X)$. We say that $S$ {\em separates the points of $X$} if, for each $x,y\in X$ with $x\neq y$, there exists $f\in S$ such that $f(x)\neq f(y)$.

\begin{definition}
Let $X$ be a compact Hausdorff space. A {\em normed function algebra} on $X$ a normed algebra $(A,\emptynorm)$ such that $A$ is a subalgebra of $C(X)$, $A$ contains all constant functions and separates the points of $X$, and, for each $f\in A$, $\norm{f}\geq\abs{f}_X$. A {\em Banach function algebra} on $X$ is a normed function algebra on $X$ which is complete. A {\em uniform algebra} is a Banach function algebra such that $\emptynorm=\emptyabs_X$.
\end{definition}

Let $X$ be a compact Hausdorff space and let $A$ be a Banach function algebra on $X$. We say that $A$ is {\em natural on $X$} if every character on $A$ is given by evaluation at some point of $X$.

We refer the reader to \cite[Chapter~4]{dales2000} for further information on Banach function algebras and uniform algebras.

We are particularly interested in Banach function algebras consisting of continuous functions on a compact plane set which satisfy some notion of differentiability.

\begin{definition}
Let $X$ be a perfect compact plane set and let $f:X\to \C$ be a function. We say that $f$ is {\em complex differentiable at $x\in X$} if the limit
\[
f'(x):=\lim_{\substack{z\to x\\z\in X}}\frac{f(z)-f(x)}{z-x}
\]
exists. We say that $f$ is {\em complex differentiable on $X$} if $f$ is complex differentiable at each point $x\in X$ and we call the function $f':X\to\C$ the {\em derivative} of $f$. We say that $f$ is {\em continuously complex differentiable} if $f'$ is continuous.
\end{definition}

In the remainder of this paper, we shall say {\em differentiable} and {\em continuously differentiable} to mean complex differentiable and continuously complex differentiable, respectively. We refer the reader to \cite{conway1978}, for example, for results from complex analysis.

Let $X$ be a perfect compact plane set. We denote the algebra of all continuously differentiable functions on $X$ by $\od1X$. For each $n\in\N$, let $\od{n}{X}$ denote the algebra of all $n$-times continuously differentiable functions on $X$ (defined inductively). Let $\od{\infty}{X}:=\bigcap_{n=1}^\infty\od{n}{X}$. Let $n\in\N$ and let $f\in \od{n}{X}$. We denote the $n$th derivative of $f$ by $f\diff n$, and we will often write $f\diff 0$ for $f$.

\begin{definition}
Let $M=(M_n)_{n=0}^\infty$ be a sequence of positive real numbers. We say that $M$ is an {\em algebra sequence} if $M_0=1$ and$,$ for each $j,k\in\No,$ we have
\[
\binom{j+k}{k}\leq\frac{M_{j+k}}{M_jM_k}.
\]
We say that $M$ is {\em log-convex} if, for each $k\in\N,$ we have $M_k^2\leq M_{k-1}M_{k+1}$.
\end{definition}

We conclude this section with a discussion of paths in $\C$. For the remainder of this section, let $a,b\in\R$ with $a<b$.

\begin{definition}
A {\em path} in $\C$ is a continuous function $\gamma:[a,b]\to\C$. Let $\gamma:[a,b]\to\C$ be a path. The {\em parameter interval of $\gamma$} is the interval $[a,b]$. The {\em endpoints of $\gamma$} are the points $\gamma(a)$ and $\gamma(b),$ which we denote by $\gamma\pin$ and $\gamma\pfi,$ respectively. We denote by $\gamma\adj$ the {\em image} $\gamma([a,b])$ of $\gamma$. A {\em subpath} of $\gamma$ is a path obtained by restricting $\gamma$ to a non-degenerate$,$ closed subinterval of $[a,b]$. If $X$ is a subset of $\C$ then we say that $\gamma$ is a {\em path in $X$} if $\gamma\adj\subseteq X$.
\end{definition}

Let $\gamma:[a,b]\to\C$ be a path in $\C$. We say that $\gamma$ is a {\em Jordan path} if $\gamma$ is an injective function. We denote the {\em length} of $\gamma$, as defined in \cite[Chapter~6]{apostol1974}, by $\pl\gamma$, and we say that $\gamma$ is {\em rectifiable} if $\pl\gamma<\infty$ and $\gamma$ is {\em non-rectifiable} otherwise. We say that $\gamma$ is {\em closed} if $\gamma\pfi=\gamma\pin$, and we say that $\gamma$ is a closed Jordan path if $\gamma$ is closed and $\gamma(s)=\gamma(t)$, where $s,t\in[a,b]$, implies that either $s=t$ or $s=a$ and $t=b$. We say that $\gamma$ is {\em admissible} if $\gamma$ is rectifiable and has no constant subpaths. The {\em reverse} of $\gamma$ is the path $-\gamma:[-b,-a]\to\gamma\adj$ given by $-\gamma(t)=\gamma(-t)$. It is standard that
\[
\int_{-\gamma}f(z)\dm z=-\int_\gamma f(z)\dm z\qquad(f\in C(\gamma\adj)).
\]
Now suppose that $\gamma$ is non-constant and rectifiable. We define the {\em path length parametrisation} $\gamma\plp:[0,\pl\gamma]\to\C$ of $\gamma$ to be the unique path satisfying $\gamma\plp(\pl{\gamma|[a,t]})=\gamma(t)$ ($t\in [a,b]$); see, for example, \cite[pp.~109-110]{federer1969} for details. We define the {\em normalised path length parametrisation} $\gamma\pln:[0,1]\to\C$ of $\gamma$ to be the path such that $\gamma\pln(t)=\gamma\plp(t\pl\gamma)$ for each $t\in[0,1]$. It is clear that $\gamma\plp$ and $\gamma\pln$ are necessarily admissible paths and $(\gamma\plp)\adj=(\gamma\pln)\adj=\gamma\adj$. It is not hard to show, using \cite[Theorem~2.4.18]{federer1969}, that
\[
\int_\gamma f(z)\dm z=\int_{\gamma\plp}f(z)\dm z=\int_{\gamma\pln} f(z)\dm z,
\]
for all $f\in C(\gamma\adj)$. We shall use this fact implicitly throughout.

\begin{definition}
Let $X$ be a perfect compact plane set and let $\F$ be a collection of paths in $X$. Define $\F\adj:=\{\gamma\adj:\gamma\in\F\}$. We say that $\F$ is {\em effective} if $\bigcup\F\adj$ is dense in $X,$ each path $\gamma\in\F$ is admissible$,$ and each subpath of a path in $\F$ belongs to $\F$.
\end{definition}

Let $X$ be a compact plane set. We say that $X$ is {\em semi-rectifiable} if the set of all Jordan paths in $X$ is an effective collection of paths in $X$. We say that $X$ is {\em rectifiably connected} if, for each $x,y\in X$, there exists a rectifiable path $\gamma\in\F$ such that $\gamma\pin=x$ and $\gamma\pfi=y$. We say that $X$ is {\em uniformly regular} if there exists a constant $C>0$ such that for each $x,y\in X$ there exists a rectifiable path $\gamma$ in $X$ with $\gamma\pin=x$ and $\gamma\pfi=y$ such that $\pl\gamma\leq C\abs{x-y}$. We say that $X$ is {\em pointwise regular} if for each $x\in X$ there exists a constant $C_x>0$ such that$,$ for each $y\in X,$ there exists a path $\gamma$ in $X$ with $\gamma\pin=x$ and $\gamma\pfi=y$ such that $\pl\gamma\leq C_x\abs{x-y}$. Note that each of the above conditions on $X$ imply that $X$ is perfect.

\section{$\F$-derivatives}
\label{Fderivs}
In this section we discuss algebras of $\F$-differentiable functions as investigated in \cite{bland2005} and \cite{dalefein2010}.

\begin{definition}
Let $X$ be a perfect compact plane set$,$ let $\F$ be a collection of rectifiable paths in $X,$ and let $f\in C(X)$. A function $g\in C(X)$ is an {\em $\F$-derivative for $f$} if$,$ for each $\gamma\in\F,$ we have
\[
\int_{\gamma}g(z)\dm z=f(\gamma\pfi)-f(\gamma\pin).
\]
If $f$ has an $\F$-derivative on $X$ then we say that $f$ is $\F$-differentiable on $X$.
\end{definition}

The following proposition is a list of the elementary properties of $\F$-differentiable functions. Details can be found in \cite{bland2005} and \cite{dalefein2010}.

\begin{proposition}\label{F derive old properties}
Let $X$ be a semi-rectifiable compact plane set and let $\F$ be an effective collection of paths in $X$.
\begin{enumerate}
  \item Let $f,g,h\in C(X)$ be such that $g$ and $h$ are $\F$-derivatives for $f$. Then $g=h$.
  \item Let $f\in\od1X$. Then the usual complex derivative of $f$ on $X,$ $f',$ is an $\F$-derivative for $f$.
  \item Let $f_1,f_2,g_1,g_2\in C(X)$ be such that $g_1$ is an $\F$-derivative for $f_1$ and $g_2$ is an $\F$-derivative for $f_2$. Then $f_1g_2+g_1f_2$ is an $\F$-derivative for $f_1f_2$.
  \item Let $f_1,f_2,g_1,g_2\in C(X)$ and $\alpha,\beta\in\C$ be such that $g_1$ is an $\F$-derivative for $f_1$ and $g_2$ is an $\F$-derivative for $f_2$. Then $\alpha g_1+\beta g_2$ is an $\F$-derivative for $\alpha f_1+\beta f_2$.
\end{enumerate}
\end{proposition}

Let $X$ be a semi-rectifiable compact plane set, and let $\F$ be an effective collection of paths in $X$. By (a) of the above proposition, $\F$-derivatives are unique. So, in this setting, we write $f\fdiff 1$ for the unique $\F$-derivative of an $\F$-differentiable function. This will be the case considered throughout the remainder of this paper. We will often write $f\fdiff 0$ for $f$. We write $\fd1X$ for the algebra of all $\F$-differentiable functions on $X$. We note that, with the norm $\norm{f}_{\F,1}:=\abs{f}_X+\abs{f\fdiff 1}_{X}$ ($f\in\fd1X$), the algebra $\fd1X$ is a Banach function algebra on $X$ (\cite[Theorem~5.6]{dalefein2010}).

For each $n\in\N$, we define (inductively) the algebra
\[
\fd nX:=\{f\in\fd{1}X:f\fdiff 1\in\fd{n-1}X\},
\]
and, for each $f\in\fd nX,$ we write $f\fdiff n$ for the $n$th $\F$-derivative for $f$. Note that, for each $n\in\N$, $\fd nX$ is a Banach function algebra on $X$ (see \cite{bland2005}) when given the norm
\[
\norm{f}_{\F,n}:=\sum_{k=0}^n\frac{\abs{f\fdiff k}_X}{k!}\qquad(f\in\fd nX).
\]
In addition, we define the algebra $\fd\infty X$ of all functions which have $\F$-derivatives of all orders; that is, $\fd\infty X=\bigcap_{n=1}^\infty\fd nX$. It is easy to see that, for each $n\in\N$, we have $\od{n}{X}\subseteq\fd nX$ and $\od{\infty}{X}\subseteq\fd\infty X$.

We now describe a class of algebras of infinitely $\F$-differentiable functions analogous to Dales-Davie algebras as introduced in \cite{bland2005} (see also \cite{chaobankoh2012}).

\begin{definition}
Let $X$ be a semi-rectifiable$,$ compact plane set and let $\F$ be an effective collection of paths in $X$. Let $M=(M_n)_{n=0}^\infty$ be an algebra sequence. We define the normed algebra
\[
\fd{}{X,M}:=\left\{f\in\fd{\infty}{X}:\sum_{j=0}^\infty\frac{\abs{f\fdiff j}_X}{M_j}<\infty\right\}
\]
with pointwise operations and the norm
\[
\norm{f}:=\sum_{j=0}^\infty\frac{\abs{f\fdiff j}_X}{M_j}\qquad(f\in\fd{}{X,M}).
\]
\end{definition}

\section{$\F$-quasianalyticity}
In this section, we discuss an $\F$-differentiability version of quasianalyticity, and give a sufficient condition for a subalgebra of $\fd{\infty}X$ new notion of quasianalyticity.

We now introduce the following notion of $\F$-quasianalyticity.

\begin{definition}\label{Fquasidef}
Let $X$ be a semi-rectifiable compact plane set and let $\F$ be an effective collection of paths in $X$. Let $A$ be a subalgebra of $\fd{\infty}{X}$. Then $A$ is {\em $\F$-quasianalytic} if, for each $\gamma\in\F$ and $z_0\in\gamma\adj$, the conditions
\[
f\in A\quad\text{and}\quad f\fdiff k(z_0)=0\quad(k=0,1,\dotsc),
\]
together imply that $f(z)=0$ for all $z\in\gamma\adj$.
\end{definition}

Let $X$ be a semi-rectifiable compact plane set and let $\F$ be an effective collection of paths in $X$.

We now aim to give some sufficient conditions for $\F$-quasianalyticity for the algebras $\fd{}{X,M}$. Our method will follow the proof in \cite{cohen1968} of the traditional Denjoy-Carleman theorem.

For the remainder of the section we fix an admissible path $\Gamma$. We also fix $\F$ to be the collection of all subpaths and reverses of subpaths of $\Gamma$. Let $M=(M_n)_{n=0}^\infty$ be a sequence of positive real numbers satisfying
\begin{equation}\label{quasisufficientsum}
\sum_{n=1}^\infty M_{n}^{-1/n}=+\infty.
\end{equation}
We write $\dm s$ for integrals with respect to the path length measure. Set $D:=\fd{\infty}{\Gamma\adj}$ and, for each $n\in\N$, set $D_n:=\fd{n}{\Gamma\adj}$.

We will require the following lemmas. The first lemma is a summary of the properties of the {\em log-convex minorant} of a sequence of positive real numbers. We refer the reader to \cite[Chapter~IV]{koosis1998} and \cite[Chapter~1]{mandelbrojt1952} for details and properties of the log-convex minorants. The properties listed below are from \cite{cohen1968}.

\begin{lemma}\label{logconvexminorant}
Let $(M_n)_{n=0}^\infty$ be a sequence of positive real numbers such that $\liminf_{n\to\infty} M_n^{1/n}=+\infty$. Then there exists a log-convex sequence $(M_n\cm)_{n=0}^\infty$ and a strictly increasing sequence $(n_j)_{j=0}^\infty$ of integers with $n_0=0$ such that$:$
\begin{enumerate}
  \item $M_k\cm\leq M_k$ for all $k\in\No;$
  \item $M_{n_k}\cm=M_{n_k}$ for all $k\in\No;$
  \item for each $k\in\No,$ we have $M_j\cm/M_{j+1}\cm=M_{n_k}\cm/M_{n_k+1}\cm$ for all $j\in\No$ with $n_k\leq j<n_{k+1}$.
\end{enumerate}
\end{lemma}

It is easy to see that if a  sequence $M=(M_n)_{n=0}^\infty$ of positive real numbers satisfies \eqref{quasisufficientsum} then the log-convex minorant $M\cm=(M_n\cm)_{n=0}^\infty$, as given by Lemma \ref{logconvexminorant}, satisfies \eqref{quasisufficientsum}. Moreover, if $M\cm$ satisfies \eqref{quasisufficientsum} then $M\cm$ satisfies $\sum_{n=0}^\infty M_n\cm/M_{n+1}\cm=\infty$; see \cite[Chapter~IV]{koosis1998} for details.

We will require lemmas to prove the main result. The first lemma is standard; see, for example, \cite[Proposition~1.17]{conway1978}.

\begin{lemma}\label{pathintderivest}
Let $\gamma\in\F$ and let $f\in D_1$. Then
\[
\abs{f(\gamma\pfi)}\leq\abs{f(\gamma\pin)}+\pl\gamma\abs{f\fdiff 1}_{\gamma\adj}.
\]
\end{lemma}

Our next lemma is an $\F$-differentiability analogue of \cite[Lemma~2]{cohen1968}.

\begin{lemma}\label{cohenlemma}
Let $\gamma\in\F,$ let $M>0,$ let $m\in\N,$ and let $\sigma=\gamma\plp$. Let $s_0,\dotsc,s_{m-1}\in[0,\pl\sigma]$ such that $0\leq s_0<s_1<\dotsb<s_{m-1}< \pl\sigma$. Let $g\in D_m$ and suppose that $\abs{g\fdiff m(\sigma(t))}\leq M$ for all $t\in[s_0,\pl\gamma]$. Then$,$ for all $s\in (s_{m-1},\pl\gamma],$ we have
\[
\abs{g(\sigma(s))}\leq\sum_{p=0}^{m-1}\frac{\abs{g\fdiff p(\sigma(s_{m-p-1}))}(s-s_{m-p-1})^p}{p!}+\frac M{m!}(s-s_0)^m.
\]
\end{lemma}
\begin{proof}
Note that, for each $a,b\in[0,\pl\gamma]$ with $a<b$, 
${\pl{\sigma|[a,b]}=b-a}$. We prove the result by induction on $m$. If $m=1$ then, by Lemma \ref{pathintderivest} applied to $\sigma|{[s_0,s]}$, we have
\[
\abs{g(z)}\leq\abs{g(z_0)}+M\pl{\sigma|{[s_0,s]}}=\abs{g(z_0)}+M(s-s_0).
\]
Let $m\in\N$ and assume that the result holds for $m-1$. Since $g\fdiff 1\in D_{m-1}$, it follows that, for each $s\in[s_{m-2},\pl\gamma]$, we have
\[
\abs{g\fdiff 1(\sigma(s))}\leq\sum_{p=0}^{m-2}\frac{\abs{g\fdiff {p+1}(\sigma(s_{m-p-2}))}(s-s_{m-p-2})^p}{p!}+\frac M{(m-1)!}(s-s_0)^{m-1}.
\]
Fix $t\in (s_{m-1},\pl\gamma]$. For each $j\in\{1,\dotsc,m-1\}$ let $\sigma_j:=\sigma|{[s_j,t]}$. Then
\begin{align}
\int_{\sigma_{m-1}}\abs{g\fdiff 1(\sigma(s))}\dm s&\leq \int_{\sigma_{m-1}}\sum_{p=0}^{m-2}\frac{\abs{g\fdiff {p+1}(\sigma(s_{m-p-2}))}(s-s_{m-p-2})^p}{p!}\dm s\nonumber\\
&\qquad{}+\int_{\sigma_{m-1}}\frac M{(m-1)!}(s-s_0)^{m-1}\dm s.\label{path length n-1 integral}
\end{align}
For each $j\in\{1,\dotsc,m-2\}$, we have
\begin{equation}\label{pathlength integral range bigger}
\int_{\sigma_{m-1}}(s-s_{m-j-2})^{j-1}\dm s\leq \int_{\sigma_{m-j-1}}(s-s_{m-j-2})^{j-1}\dm s=\frac{(s-s_{m-j-2})^{j}}{j!},
\end{equation}
and, by combining \eqref{pathlength integral range bigger} and \eqref{path length n-1 integral}, we obtain
\begin{equation}
\int_{\sigma_{m-1}}\abs{g\fdiff 1(\sigma(s))}\dm s\leq\sum_{p=1}^{m-1}\frac{\abs{g\fdiff {p}(\sigma(s_{m-p-1}))}(s-s_{m-p-1})^p}{p!}+\frac M{m!}(s-s_0)^{m}.\label{path length integrated path length n-1}
\end{equation}
But now
\[
\abs{g(\sigma(t))}\leq\abs{g(z_{m-1})}+\int_{\sigma_{m-1}}\abs{g\fdiff 1(\sigma(s))}\dm s,
\]
and combining this with \eqref{path length integrated path length n-1}, we obtain the desired result.
\end{proof}

We now check an easy special case of our result.

\begin{lemma}\label{easycase}
Let $M=(M_n)_{n=0}^\infty$ be a sequence of positive real numbers such that $\liminf_{n\to\infty}M_n^{1/n}<\infty$. Let $z_0\in\Gamma\adj,$ and let $f\in D$ with $\abs{f\fdiff k}_{\Gamma\adj}\leq M_k$ and $f\fdiff k(z_0)=0$ for all $k\in\No$. Then $f(z)=0$ for all $z\in\Gamma\adj$.
\end{lemma}
\begin{proof}
Fix $z\in \Gamma\adj$. Let $\gamma\in\F$ such that $\gamma\pin=z_0$, $\gamma\pfi=z$, and $\gamma=\gamma\plp$. We first {\em claim} that, for each $k,n\in\No$ with $k\leq n$, we have
\begin{equation}
\abs{f\fdiff k(z)}\leq M_n\frac{s^{n-k}}{(n-k)!},
\end{equation}
where $s=\pl\gamma$.

Fix $n\in\No$.
We prove our claim by induction on $k\leq n$. First suppose that $k=n$ so that $n-k=0$. Let $z\in\gamma\adj$. We have
\[
\abs{f\fdiff{k-1}(z)}\leq\abs{f\fdiff{k-1}(\gamma\pfi)}+\int_\gamma\abs{f\fdiff k}_{\gamma\adj}\dm s=\int_{\gamma}\abs{f\fdiff k}_{\gamma\adj}\dm s,
\]
and applying the claim to $\abs{f\fdiff k(\zeta)}$, we have
\[
\abs{f\fdiff{k-1}(z)}\leq\int_{\gamma}\frac{M_ns^{n-k}}{(n-k)!}\dm s=\frac{M_ns^{n-k+1}}{(n-k+1)!}.
\]
This proves the claim.

We now see that $\abs{f(z)}\leq M_n s^n/n!$ for all $n\in\No$ and all $z\in\gamma\adj$. Now since $\liminf_{n\to\infty}M_n^{1/n}<\infty$, there exists $R>0$ such that $M_n<R^n$ for infinitely many $n\in\N$. Let $(n_k)_{k=1}^\infty$ be a strictly increasing sequence in $\N$ such that $M_{n_k}<R^{n_k}$ for all $k\in\N$. Then, for each $z\in\gamma\adj$, we have
\[
\abs{f(z)}\leq \frac{M_{n_k}s^{n_k}}{n_k!}<\frac{(Rs)^{n_k}}{n_k!}\to 0
\]
as $k\to\infty$. This holds for all $z\in\Gamma\adj$, so the result follows.
\end{proof}

Let $0<\alpha<1$. As in \cite{cohen1968}, we define $B_{j,k}\diff\alpha$, for $j,k\in\No$ with $k\geq j\geq 0$, as follows. For each $j\in\No$, let $B_{0,j}^{(\alpha)}=0$ and, for each $j\in\N$, let $B_{j,j}^{(\alpha)}=1$. For each $j,k\in\No$ with $k>j$, define $B_{j+1,k+1}^{(\alpha)}$ inductively by setting
\[
B_{j+1,k+1}^{(\alpha)}=B_{j,k+1}^{(\alpha)}+\alpha B_{j+1,k}^{(\alpha)}.
\]

Our main tool in the proof of the main theorem is the following lemma, which can be distilled from the proof of \cite[Lemma~1]{cohen1968} and Stirling's approximation. We omit the proof.

\begin{lemma}\label{cohenconstbounds}
Let $\alpha\in(0,1/4\re)$. Then there exists a constant $K>0$ such that $B_{j,k}\diff\alpha\leq K\alpha<1/2$ for all $j,k\in\N$ with $j<k$. Moreover$,$ for each $n\in\N,$ we have
\[
\sum_{j=0}^{n-1}\frac{((j+1)\alpha)^j}{j!}+2\frac{(n\alpha)^n}{n!}<2.
\]
\end{lemma}

Note that, if $\alpha\in(0,1/4\re)$, then
\begin{equation}
  B_{j+1,k+1}\diff\alpha\geq B_{j,k+1}\diff\alpha\label{B const inequality}
\end{equation}
for all $j,k\in\No$ with $k\geq j$.

We now state and prove our main result. The proof is essentially the one used in \cite{cohen1968}, adapted for $\F$-differentiation, and including additional details for the convenience of the reader.

\begin{theorem}\label{newsufficiency}
Let $\Gamma:[a,b]\to\C$ be an admissible path and let $\F$ denote the collection of all subpaths and reverses of subpaths of $\Gamma$. Let $(M_n)_{n=0}^\infty$ be a sequence of positive real numbers such that $\sum_{n=1}^\infty M_n^{-1/n}=\infty$. Suppose that $f\in\fd{\infty}{\Gamma\adj}$ and $z_0\in\Gamma\adj$ satisfy
\begin{equation}\label{sufficiency assumptions on f}
\abs{f\fdiff k}_{\Gamma\adj}\leq M_k\quad\text{and}\quad f\fdiff k(z_0)=0
\end{equation}
for all $k\in\No$. Then $f(z)=0$ for all $z\in\Gamma\adj$.
\end{theorem}
\begin{proof}
If $\liminf_{n\to\infty}M_n^{1/n}<\infty$ then the result follows from Lemma \ref{easycase}, so suppose that $\liminf_{n\to\infty} M_n^{1/n}=\infty$. By Lemma \ref{logconvexminorant}, there exists a log convex sequence $(M_n\cm)_{n=0}^\infty$ of positive real numbers and a strictly increasing sequence $(n_j)_{j=0}^\infty$ with $n_0=0$ such that:
\begin{enumerate}
  \item $M_k\cm\leq M_k$ for all $k\in\No;$
  \item $M_{n_k}\cm=M_{n_k}$ for all $k\in\No;$
  \item for each $k\in\No,$ we have $M_j\cm/M_{j+1}\cm=M_{n_k}\cm/M_{n_k+1}\cm$ for all $j\in\No$ with $n_k\leq j<n_{k+1}$.
\end{enumerate}
By the comments following Lemma \ref{logconvexminorant}, we have
\[
\sum_{n=0}^\infty\frac{M_n\cm}{M_{n+1}\cm}=\infty.
\]
Let $z\in\Gamma\adj$ and let $\gamma\in\F$ such that $\gamma\pin=z_0,$ $\gamma\pfi=z$ and $\sigma=\gamma\plp$.

Fix $n\in\N$ such that $n$ is an element in the sequence $(n_p)_{p=0}^\infty$ and such that ${\sum_{j=0}^n M_{j}\cm/M_{j+1}\cm>4\re\pl\sigma}$. Let ${\alpha:=\pl\sigma/(\sum_{j=0}^n M_j\cm/M_{j+1}\cm)}$. Then we have ${\alpha\in(0,1/4\re)}$ and so, by Lemma \ref{cohenconstbounds}, $B_{u,v}\diff \alpha<1/2$ for all $u,v\in\N$ with $v>u>0$.

Define the points $0=x_0<x_1<\dotsb<x_n\leq \pl\sigma$ such that
\[
x_j-x_{j-1}=\alpha M_{n-j}\cm/M_{n-j+1}\cm\qquad j=1,2,\dotsc,n.
\]

For each $j\in\{0,\dotsc,n\}$ we {\em claim} that, for each $k\in\No$, with $k\leq n-j+1$,
\begin{equation*}\label{sufficiency proof claim}
\abs{f\fdiff k(\sigma(t))}\leq B_{j,n-k+1}\diff\alpha M_k\cm\qquad(t\in[0,x_j]).\tag{$\clstar jk$}
\end{equation*}

The proof of the claim is by induction on $j$. Since $f\fdiff k(\sigma\pin)=0$ for each $k\in\No$, \refcl 0k holds for all $k\in\No$ with $k\leq n+1$.

Fix $j\in\{1,\dotsc,n\}$. Assume now that \refcl{j'}{k'} holds for all $j',k'\in\No$ with $j'<j$ and $k'\leq n-j'+1$. Set $i:=n-j+1$. We now prove \refcl{j}{k} holds for each $k\in\No$ with $k\leq i$ by backwards induction on $k$. We first check the base case. Suppose that $k=i$. If $i=n_r$ for some $r\in\No$, then $B_{j,j}\diff\alpha=1$ and $\abs{f\fdiff{i}(\sigma(t))}\leq M_{i}= M_{i}\cm$ for all $t\in[0,x_{j}]$ by \eqref{sufficiency assumptions on f}, and so \refcl jk holds.

Otherwise, $i\neq n_r$ for all $r\in\No$, in which case there exists $r\in\No$ such that $n_r<i<n_{r+1}\leq n$. For each $s\in[0,x_{j-1}]$, by \refcl{j-1}{i} and \eqref{B const inequality}, we have
\[
\abs{f\fdiff i(\sigma(s))}\leq B_{j-1,j}\diff\alpha M_{i}\cm\leq B_{j,j}\diff\alpha M_{i}\cm=M_{i}\cm,
\]
and so it remains to show that $\abs{f\fdiff i(\sigma(s))}\leq M_{i}\cm$ for all $s\in[x_{j-1},x_j]$.

As in \cite{cohen1968}, let ${m:=n_{r+1}-i}$ and let ${R:=M_{n_r}\cm/M_{n_r+1}\cm}$. Note that
\[
j-m=n-n_{r+1}+1\geq 1
\]
and, for each ${p\in\No}$ with ${p\leq m}$, we have $M_{i}\cm=M_{i+p}\cm R^p$. Also, for each ${\ell\in\N}$ with ${n_r\leq\ell<n_{r+1}}$, we have ${M_\ell\cm/M_{\ell+1}\cm=R}$ and so ${x_{j-p}-x_{j-p-1}=\alpha R}$ for each $p\in\N$ with $p\leq m$.  In particular, $M_i\cm=M_{n_{r+1}}\cm R^m$. For each $s\in[x_{j-1},x_j]$, by Lemma \ref{cohenlemma} (applied with $M=M_{n_{r+1}}\cm=M_{n_{r+1}}$, $g=f\fdiff i$ and points $s_0=x_{j-m},\dotsc,s_{m-1}=x_{j-1}$), we have
\begin{align*}
\abs{f\fdiff{i}(\sigma(s))}&\leq\sum_{p=0}^{m-1} \frac{\abs{f\fdiff{i+p}(\sigma(x_{j-p-1}))}(s-x_{j-p-1})^p}{p!}+\frac {M_{n_{r+1}}\cm}{m!}(s-x_{j-m})^m\\
&\leq \sum_{p=0}^{m-1}\frac{\abs{f\fdiff{i+p}(\sigma(x_{j-p-1}))}((p+1)\alpha R)^p}{p!}+\frac{M_{n_{r+1}}\cm}{m!}(m\alpha R)^m
\end{align*}
and, by applying \refcl{j-p-1}{i+p} for each $0\leq p\leq m-1$, we obtain
\begin{align*}
\abs{f\fdiff{i}(\sigma(s))}&\leq \sum_{p=0}^{m-1}\frac{B_{j-1-p,n-(i+p)+1}\diff\alpha M_{i+p}\cm((p+1)\alpha R)^p}{p!}+\frac{M_{n_{r+1}}\cm}{m!}(m\alpha R)^m\\
&=M_{i}\cm\left(\sum_{p=0}^{m-1}\frac{B_{j-1-p,j-p}\diff\alpha((p+1)\alpha)^p}{p!} +\frac{(m\alpha)^m}{m!}\right),
\end{align*}
Since $B_{u,v}\diff\alpha<1/2$ for all $u,v\in\No$ with $v>u>0$, we have
\[
\abs{f\fdiff i(\sigma(s))}\leq \frac{M_{i}\cm}{2}\left(\sum_{p=0}^{m-1}\frac{((p+1)\alpha)^p}{p!} +2\frac{(m\alpha)^m}{m!}\right),
\]
and so, by Lemma \ref{cohenconstbounds}, we have $\abs{f\fdiff i(\sigma(s))}<M_{i}\cm$. This concludes the proof of the base case $k=i$.

Now let $k\in\No$ with $k\leq i-1$ and assume that \refcl{j}{k+1} holds, i.e.,
\[
\abs{f\fdiff{k+1}(\sigma(t))}\leq B_{j,n-(k+1)+1}\diff\alpha M_{k+1}\cm
\]
for all $t\in[0,x_j]$. Let $s\in[0,x_j]$. If $s\in[0,x_{j-1}]$ then, by applying \refcl{j-1}{k} and \eqref{B const inequality}, we have
\[
\abs{f\fdiff k(\sigma(s))}\leq B_{j-1,n-k+1}\diff\alpha M_{k}\cm\leq B_{j,n-k+1}\diff\alpha M_{k}\cm.
\]
Thus we may assume that $s\in[x_{j-1},x_j]$. By Lemma \ref{pathintderivest}, we have
\[
\abs{f\fdiff k(\sigma(s))}\leq\abs{f\fdiff k(\sigma(x_{j-1}))}+(x_j-x_{j-1}) \sup\{\abs{f\fdiff{k+1}(\sigma(t))}:t\in[x_{j-1},x_j]\}.
\]
Applying \refcl{j-1}{k} to the first term and applying \refcl{j}{k+1} to the second term we obtain
\[
\abs{f\fdiff k(\sigma(s))}\leq B_{j-1,n-k+1}\diff\alpha M_k\cm+\alpha\frac{M_{n-j}\cm}{M_{n-j+1}\cm}B_{j,n-k}\diff\alpha M_{k+1}\cm.
\]
Since $M\cm$ is log-convex and $k\leq n-j$, we have $M_{n-j}\cm/M_{n-j+1}\cm\leq M_k\cm/M_{k+1}\cm$, so we obtain
\[
\abs{f\fdiff k(\sigma(s))}\leq M_k\cm\left(B_{j-1,n-k+1}\diff\alpha+\alpha B_{j,n-k}\diff\alpha\right)=B_{j,n-k+1}\diff\alpha M_k\cm.
\]
Thus \refcl jk holds, and both inductions may now proceed.
\smallskip

Now, by Lemma \ref{cohenconstbounds}, there exists a constant $K>0$ such that, for all $u,v\in\N$ with $v>u$, we have $B_{u,v}<\alpha K$. Thus $\abs{f(\sigma(t))}\leq K\alpha M_0$ for all $t\in[0,\pl\sigma]$. It follows that $f(\sigma(t))=0$ for all $t\in[0,\pl\sigma]$.In particular, $\abs{f(z)}=0$ and hence $f(z)=0$. Since $z\in\Gamma\adj$ was arbitrary, the above holds for all $z\in\Gamma\adj$. This completes the proof.
\end{proof}

In the remainder of this paper we adopt the following convention. Let $X$ be a semi-rectifiable compact plane set, let $\F$ be an effective collection of paths in $X$, and let $f\in\fd{\infty}{X}$. If there exists $k\in\No$ such that $\abs{f\fdiff k}_X=0$ then we write
\[
\sum_{j=1}^\infty \abs{f\fdiff j}_X^{-1/j}=+\infty.
\]

Our first corollary will be used in the next section.

\begin{corollary}\label{Useful Den Carl}
Let $\gamma$ be an admissible path in $\C,$ let $\F$ be an effective collection of paths in $\gamma\adj,$ and let $f\in\fd\infty{\gamma\adj}$. Suppose that
\begin{equation}\label{quasianaylticcrit1}
\sum_{j=1}^\infty\abs{f\fdiff j}^{-1/j}=+\infty.
\end{equation}
If there exists $z_0\in\gamma\adj$ such that $f\fdiff k(z_0)=0$ for all $k\in\No$ then $f$ is identically zero on $\gamma\adj$.
\end{corollary}
\begin{proof}
For each $k\in\No,$ set $M_k=\abs{f\fdiff k}_{\gamma\adj}$. By \eqref{quasianaylticcrit1}, the sequence $M=(M_k)_{k=0}^\infty$ satisfies \eqref{quasisufficientsum} and certainly $\abs{f\fdiff k}_{\gamma\adj}\leq M_k$ for all $k\in\No$. Suppose that there exists $z_0\in\gamma\adj$ such that $f\fdiff k(z_0)=0$ for all $k\in\No$. Then, by Theorem \ref{newsufficiency}, we have $f(\gamma\adj)\subseteq\{0\}$. This completes the proof.
\end{proof}

Our next corollary asserts the existence of an $\F$-quasianalytic algebra of the form $\fd{}{X,M}$.

\begin{corollary}
Let $X$ be a semi-rectifiable compact plane set$,$ let $\F$ be an effective collection of paths in $X,$ and let $M=(M_n)_{n=0}^\infty$ be an algebra sequence which satisfies $\eqref{quasisufficientsum}$. Then $\fd{}{X,M}$ is $\F$-quasianalytic.
\end{corollary}
\begin{proof}
Fix $f\in\fd{}{X,M}$. Then there exists $N\in\N$ large enough so that ${\abs{f\fdiff k}_X\leq M_k}$ for all ${k\in\No}$ with ${k\geq N}$. If there exists ${k\in\No}$ such that ${\abs{f\fdiff k}_X=0}$, then \eqref{quasianaylticcrit1} holds, so suppose that ${\abs{f\fdiff k}_X>0}$ for all ${k\in\No}$. Then we have
\[
\sum_{j=1}^\infty \abs{f\fdiff j}_X^{-1/j}\geq \sum_{j=N}^\infty\abs{f\fdiff j}_X^{-1/j}\geq \sum_{j=N}^\infty M_j^{-1/j}=\infty.
\]
Thus, in either case, $f$ satisfies \eqref{quasianaylticcrit1}. Thus, by Corollary \ref{Useful Den Carl}, if there exist $\gamma\in\F$ and $z_0\in\gamma\adj$ such that $f\fdiff k(z_0)=0$ for all $k\in\No$, then $f(\gamma\adj)\subseteq\{0\}$. It follows that $\fd{}{X,M}$ is $\F$-quasianalytic.
\end{proof}

Since $\od{\infty}{X}\subseteq\fd{\infty}{X}$, Theorem \ref{newsufficiency} generalises \cite[Theorem~1.11]{dales1973} to obtain the following.

\begin{corollary}
Let $M=(M_n)_{n=0}^\infty$ be a sequence of positive real numbers such that $\sum_{n=1}^\infty M_{n}^{-1/n}=\infty$. Let $\gamma$ be an admissible path in $\C$. Suppose that $f\in\od{\infty}{X},$ $f\diff k(\gamma\pin)=0$ and $\abs{f\diff k}\leq M_k$ for $k=0,1,2,\dotsc$. Then $f$ is identically zero on $\gamma\adj$.
\end{corollary}

We conclude this section with a note about $\F$-analyticity, as introduced in \cite{chaobankoh2012} (see also \cite{FM2015}). Let $X$ be a semi-rectifiable compact plane set, let $\F$ be an effective collection of paths in $X$, and let $f\in\fd{\infty}{X}$. We say that $f$ is {\em $\F$-analytic} if
\[
\limsup_{k\to\infty}\Biggl(\frac{\abs{f\fdiff k}}{k!}\Biggr)^{1/k}<\infty.
\]
This is a generalisation of the term {\em analytic} used in \cite{feinstein2004a,feinstein2004k}, and is used to find sufficient conditions for maps to induce homomorphisms between the algebras $\fd{}{X,M}$. (Note that, in \cite{feinstein2000a}, the term {\em analytic} was used for those functions on $X$ which extend to be analytic on a neighbourhood of $X$. This condition is stronger than in \cite{feinstein2004a,feinstein2004k}.)

Let $X$ be a semi-rectifiable compact plane set, let $\F$ be an effective collection of paths in $X$, and let $f\in\fd{\infty}{X}$. Using Theorem~\ref{newsufficiency}, we can show that if $f$ is $\F$-analytic then, for each $\gamma\in\F$ and $z\in\gamma\adj$, there exist $r>0$ and an analytic (in the usual sense) function $g:\ob{z}{r}\to\C$ such that $g|(\gamma\adj\cap\ob zr)=f|(\gamma\adj\cap\ob zr)$. From this, it follows that in fact, for each $\gamma\in\F$, there exist an open neighbourhood $U$ of $\gamma\adj$ and an analytic function $h:U\to\C$ such that $h|\gamma\adj=f|\gamma\adj$. We wish to thank Prof. J. K. Langley and Dr. D. A. Nicks for showing us how to prove the latter implication.

\section{Trivial Jensen measures without regularity}
\label{construction section}
We conclude the paper with an application of the results from the previous sections. We construct a locally connected compact plane set $X$ and an essential uniform algebra $A$ on $X$ such that $A$ does not admit any non-trivial Jensen measures but is not regular. This example will improve an example of the first author (\cite{feinstein2001}).

We begin with the relevant definitions.

\begin{definition}
Let $X$ be a compact Hausdorff space$,$ let $A$ be a uniform algebra on $X,$ and let $\varphi$ be a character on $A$. A probability measure $\mu$ on $X$ is a Jensen measure for $\varphi$ $($with respect to $A)$ if
\[
\log{\abs{\varphi(f)}}\leq \int_X\log{\abs{f}}\dm \mu\qquad(f\in A).
\]
We say that $A$ is {\em regular on $X$} if$,$ for each closed set $E\subseteq X$ and each point $x\in X\setminus E,$ there exists $f\in A$ such that $f(x)=0$ and $f(E)\subseteq\{0\}$. We say that $A$ is regular if the Gelfand transform of $A$ is regular on the character space $\Phi_A$ of $A$. We say that $A$ is {\em essential} if there exists no proper closed subset $E$ of $X$ such that $A$ contains every $f\in C(X)$ such that $f(y)=0$ for all $y\in E$.
\end{definition}

In the above definition we adopt the convention that $\log(0)=-\infty$. Let $X,A,\varphi$ be as in the above definition. It is standard that every Jensen measure for $\varphi$ is a representing measure for $\varphi$. Moreover, for each $\varphi\in\Phi_A$, there is a Jensen measure on $X$ for $\varphi$. Note that, for $x\in X$, the point-mass measure $\delta_x$ is a Jensen measure for $\varepsilon_{x}$, where (here, and for the remainder of the section) $\varepsilon_x$ is the {\em evaluation character at $x$ $($with respect to $A)$}. We say that a Jensen measure $\mu$ on $X$ for $\varepsilon_x$ is {\em trivial} if $\mu=\delta_x$.

Let $X$ be a compact plane set. Let $R_0(X)$ denote the set of restrictions to $X$ of rational functions with no poles on $X$. Let $R(X)$ denote the uniform closure of $R_0(X)$ in $C(X)$. It is standard that $R_0(X)$ is an algebra and that $R(X)$ is a natural uniform algebra on $X$. For the remainder of this section, all Jensen measures will be with respect to $R(X)$ unless otherwise specified.

For further details on uniform algebras, Jensen measures, and related topics, see \cite{browder1969,gamelin1984,dales2000,gamelin1978,stout1971}.

Let $x\in X$. Let $J_x$ denote the ideal in $R(X)$ of all functions which vanish on a neighbourhood of $x$. Let $M_x$ denote the ideal in $R(X)$ of all functions which vanish at $x$. Clearly $J_x\subseteq M_x$. We say that $x$ is a {\em point of continuity} (for $R(X)$) if, for all $y\in X\setminus\{x\}$ we have $J_y\nsubseteq M_x$. We say that $x$ is an {\em $R$-point} if, for all $y\in X\setminus\{x\}$, we have $J_x\nsubseteq M_y$. For further information see \cite{feinstein2000,feinstein2001,feinstein2012,FY2015}. (Note that in \cite{feinstein2012} points of continuity are referred to as {\em regularity points of type one} and $R$-points are referred to as {\em regularity points of type two}.)

It is standard that $R(X)$ is regular if and only if every point of $X$ is a point of continuity, and this holds if and only if every point of $X$ is an $R$-point. It is also standard that if $x$ is a point of continuity then the only Jensen measure for $\varepsilon_x$ is the point mass measure.

Let $X$ be a topological space and let $E$ be a subset of $X$. We denote by $\tint_{X}(E)$ the interior of $E$ with respect to the topological space $X$. In particular, if $E\subseteq\C$ then $\tint_{\C}E$ coincides with the usual interior of $E$.

For the remainder of this section, we denote the set of non-negative real numbers by $\Rp$. Let $X$ be a metric space, let $x\in X$, and let $r>0$. We denote the open ball in $X$ with centre $x$ and radius $r$ by $\ob[X]{x}{r}$. We the denote the corresponding closed ball by $\cb[X]{x}{r}$.

In the special case where $X=\C$, for each ${a\in\C}$ and ${r>0}$, we write ${\ob ar=\ob[\C]ar}$ and ${\cb ar=\cb[\C]ar}$. For each ${a\in\C}$, we set ${\ob a0=\emptyset}$ and ${\cb a0=\{a\}}$.

\begin{lemma}\label{subset interior points of continuity}
Let $X,Y$ be compact plane sets with $Y\subseteq X$. Suppose that $R(Y)$ is regular. Then each point $y\in\tint_X{(Y)}$ is a point of continuity for $R(X)$.
\end{lemma}
\begin{proof}
Let $y\in\tint_X{(Y)}$, and let $x\in X$ with $x\neq y$. Then there exists $r>0$ such that $\cb[X]{y}{r}\subseteq Y$. Choose $\delta\in(0,r/3)$ such that $\abs{y-x}>2\delta$. Since $R(Y)$ is regular, it follows (from \cite[lemma~1.6]{feinheath2010}, for example) that $R(Y\cap\cb{y}{r})$ is regular. Set $E:=\cb[X]{y}{r}\setminus\ob[X]{y}{\delta}$. Then there exists a function $g\in R(Y\cap\cb[X] yr)$ such that $f(y)=1$ and $f(x)=0$ for all $x\in E$. Let $f\in C(X)$ be given by $f(z)=g(z)$ for all $z\in X\cap\cb[X] yr$ and $f(z)=0$ for all $z\in X\setminus\cb[X] yr$. It follows from \cite[Corollary~II.10.3]{gamelin1983} that $f\in R(X)$ and clearly $f$ vanishes on a neighbourhood of $x$. Thus $J_{x}\nsubseteq M_y$ (where these are the ideals in $R(X)$). It follows that $y$ is a point of continuity for $R(X)$ and so the proof is complete.
\end{proof}

The following is effectively \cite[Lemma~3.1]{FY2015}.

\begin{lemma}\label{punctured neighbourhood R point}
Let $X$ be a compact plane set and let $x\in X$. Suppose that there exists a neighbourhood $U$ of $x$ in $X$ such that every point in $U\setminus\{x\}$ is a point of continuity. Then $x$ is an $R$-point.
\end{lemma}

As in \cite{FMY2014a} (see also \cite{FMY2014b}), we define abstract Swiss cheeses as follows.

\begin{definition}
An {\em abstract Swiss cheese} is a sequence $\mathbf{A}=((a_n,r_n))_{n=0}^\infty$ of elements of $\C\times\Rp$. Let $\mathbf{A}=((a_n,r_n))_{n=0}^\infty$ be an abstract Swiss cheese. Set
\[
X_{\mathbf A}:=\cb{a_0}{r_0}\setminus\left(\bigcup_{n=1}^\infty\ob{a_n}{r_n}\right),
\]
and set $\rho(\mathbf A)=\sum_{n=1}^\infty r_n$. We say that $\mathbf A$ is {\em classical} if $\rho(\mathbf A)<\infty,$ $r_0>0$ and for all $k\in \N$ with $r_k>0$ the following hold$:$
\begin{enumerate}
\item $\cb{a_k}{r_k}\subseteq\ob{a_0}{r_0};$
\item whenever $\ell\in\N$ with $r_\ell>0$ and $\ell\neq k,$ we have $\cb{a_k}{r_k}\cap\cb{a_\ell}{r_\ell}=\emptyset$.
\end{enumerate}
\end{definition}

We say that a compact plane set $X$ is a {\em Swiss cheese set} if there exists an abstract Swiss cheese $\mathbf A$ such that $X=X_{\mathbf A}$. We say that a Swiss cheese set $X$ is {\em classical} if there exists a classical abstract Swiss cheese $\mathbf A$ with $X=X_{\mathbf A}$.
If $X$ is a classical Swiss cheese set then $X$ is a uniformly regular (see the proof of \cite[Theorem~8.3]{dalefein2010}) and $R(X)$ is essential (see \cite[p.~167]{browder1969} or \cite[Theorem~1.8]{feinheath2010}). It follows that if $\mathbf A$ is classical then $X_{\mathbf A}$ is also connected and locally connected.

In \cite{feinstein2001}, the first author gave an example of a non-classical Swiss cheese set $X$ such that $R(X)$ has no non-trivial Jensen measures, but such that $R(X)$ is not regular. We shall show that there is a classical Swiss cheese set with these properties; this is the content of the following theorem, which is the main theorem of this section.

\begin{theorem}\label{trivial Jensen cheese main theorem}
There exists a classical abstract Swiss cheese $\mathbf A=((a_n,r_n))$ such that $R(X_{\mathbf A})$ is not regular and does not admit any non-trivial Jensen measures.
\end{theorem}

Most of the remainder of this section is devoted to the proof of this theorem. We require some preliminary results. The following proposition is \cite[Lemma~4.1]{FY2015}.

\begin{proposition}\label{non-trivial representing measure criteria}
Let $X$ be a compact plane set$,$ let $Y$ be a non-empty closed subset of $X,$ and let $x\in Y$. Suppose that no bounded component of $\C\setminus Y$ is contained in $X,$ and that there exists a non-trivial representing measure $\mu$ for $\varepsilon_x$ with respect to $R(X)$ such that $\supp\mu\subseteq Y$. Then $\mu$ is a non-trivial representing measure for $\varepsilon_x$ with respect to $R(Y),$ and $R(Y)\neq C(Y)$.
\end{proposition}

Note that, if $\tint_{\C}X=\emptyset$ then the condition on bounded components of $\C\setminus Y$ is automatically satisfied.

Let $X$ be a compact plane set with $\tint_{\C} X=\emptyset$, let $x\in X$, and let $\mu$ be a non-trivial representing measure for $\varepsilon_x$. Let $Y=\supp\mu\cup\{x\}$, where $\supp\mu$ denotes the closed support of $\mu$. Then, by the above, we must have $R(Y)\neq C(Y)$. In particular, as noted in \cite{FY2015}, $Y$ must have positive area. (See also the Hartogs-Rosenthal theorem \cite[Corollary~II.8.4]{gamelin1984}.) Combining these observations with Proposition \ref{non-trivial representing measure criteria} gives the following corollary, which we use below.

\begin{corollary}\label{jensen measures supports corollary}
Let $X$ be a compact plane set with $\tint_{\C}(X)=\emptyset,$ let $E$ be a closed subset of $X,$ and let $x\in E$. Suppose that $E$ has area $0,$ that $\mu$ is a Jensen measure for $\varepsilon_x$ with respect to $R(X),$ and $\mu$ is supported on $E$. Then $\mu$ is trivial.
\end{corollary}

We also require the following lemma, which is a special case of \cite[Lemma~2.1]{FY2015}.

\begin{lemma}\label{non R points in support}
Let $X$ be a compact plane set and let $x\in X$. Suppose that $\mu$ is a non-trivial Jensen measure for $x,$ and let $F$ be the closed support of $\mu$. Then$,$ for all $y\in F\setminus\{x\},$ we have $J_y\subseteq M_x$. Thus $x$ is not a point of continuity and no point of $F\setminus\{x\}$ is an $R$-point.
\end{lemma}

The following estimates on derivatives are standard. See, for example \cite[Lemma~4.4]{FY2015}. (This result also appears in \cite{feinstein2001} but with some typographical errors.)

\begin{lemma}\label{R(X)_cheese_estimates}
Let $\mathbf A=((a_n,r_n))_{n=1}^\infty$ be an abstract Swiss cheese$,$ and let $z\in\C$. For each $n\in\N,$ let $d_n$ denote the distance from $\ob{a_n}{r_n}$ to $z$. Let $d_0=r_0-\abs{z-a_0}$. Suppose that $d_n>0$ for all $n\in\No$. Then $z\in X_{\mathbf A}$ and$,$ for all $f\in R_0(X_{\mathbf A})$ and $k\in\No,$ we have
\[
\abs{f\diff k(z)}\leq k!\sum_{j=0}^\infty\frac{r_j}{d_j^{k+1}}\abs{f}_{X_{\mathbf A}}.
\]
\end{lemma}

Our construction will use the following proposition, which is a combination of \cite[Lemma~8.5]{FMY2014a} and, for example, \cite[Example~2.9]{feinheath2010}.

\begin{proposition}\label{regular annulus}
Let $a\in\C,$ $\lambda_1\geq 0,$ $\lambda_0>\lambda_1,$ and $\varepsilon>0$. Then there exists a classical abstract Swiss cheese $\mathbf A=((a_n,r_n))_{n=0}^\infty$ such that ${a_0=a_1=a},$ ${r_0=\lambda_0}$ and ${r_1=\lambda_1}$ such that ${R(X_{\mathbf A})}$ is regular and ${\sum_{n=2}^\infty r_n<\varepsilon}$.
\end{proposition}

Note that, since $R(X_{\mathbf A})$ is regular, we must have $\tint_{\C}X_{\mathbf A}=\emptyset$.

We now give the details of the construction.

\begin{lemma}\label{construction}
Let ${0<r<1}$ be given$,$ let $C_r$ denote the circle of radius $r$ centred at $0,$ and let ${\varepsilon>0}$. Then there exists a classical abstract Swiss cheese $\mathbf A=((a_n,r_n))$ such that
\begin{enumerate}
  \item $\rho(\mathbf A)<\varepsilon,$ $\tint_{\C}(X_{\mathbf A})=\emptyset,$ and $C_r\subseteq X_{\mathbf A},$
  \item there is a dense open subset $U$ of $X_{\mathbf A}$ such that $X_{\mathbf A}\setminus U$ has area zero and each point $z\in U$ is a point of continuity for $R(X_{\mathbf A}),$
  \item for each $f\in R(X),$ we have $f|C_r\in\od{\infty}{C_r}$ and $\sum_{k=1}^\infty\abs{f\diff k}_{C_r}^{-1/k}=\infty$.
\end{enumerate}
\end{lemma}
\begin{proof}
Our abstract Swiss cheese $\mathbf A$ will be obtained by combining a certain pair of sequences $(\mathbf A_n)$, $(\mathbf B_n)$ of abstract Swiss cheeses in a suitable way. We first construct the sequences $(\mathbf A_n)$, $(\mathbf B_n)$.

Choose a positive integer $n_0$ large enough so that $r+2^{1-n_0}<1$ and $r-2^{1-n_0}>0$. As in \cite{FY2015}, choose a sequence $(\gamma_n)$ of positive real numbers such that, for each $n,k\in\N$, we have
\[
\gamma_n\leq \frac{(2^{1-n_0-n})^{k+1}(\log{(k+3)})^k}{2^k},
\]
and such that $\sum_{n=1}^\infty \gamma_n<\varepsilon$. Thus, for each $k\in\N$,
\[
\sum_{n=1}^\infty\frac{\gamma_n}{(2^{1-n_0-n})^{k+1}}\leq (\log{(k+3)})^k.
\]

Let ${\mathbf A_1=((a_n^{(1)},r_n^{(1)}))}$ be the classical abstract Swiss cheese obtained from Proposition~\ref{regular annulus} applied with ${\lambda_0=r-2^{1-n_0}}$, ${\lambda_1=0}$, ${a=0}$ and ${\varepsilon=\gamma_1/2}$. Let ${\mathbf B_1=((b_n^{(1)},r_n^{(1)}))}$ be the classical abstract Swiss cheese obtained from Proposition~\ref{regular annulus} applied with ${a=0}$, ${\lambda_0=1}$ and ${\lambda_1=r+2^{1-n_0}}$ and ${\varepsilon=\gamma_1/2}$.

For each $k\in\N$ with $n\geq 2$: let ${\mathbf A_k=((a_n^{(k)},r_n^{(k)}))}$ be the classical abstract Swiss cheese obtained from Proposition~\ref{regular annulus} applied with ${\lambda_0=r-2^{2-n_0-k}}$, ${\lambda_1=r-2^{3-n_0-k}}$, ${a=0}$ and ${\varepsilon=\gamma_k/2}$; let ${\mathbf B_k=((b_n^{(k)},r_n^{(k)}))}$ be the classical abstract Swiss cheese obtained from Proposition~\ref{regular annulus} applied with ${a=0}$, ${\lambda_0=r+2^{3-n_0-k}}$ and ${\lambda_1=r+2^{2-n_0-k}}$ and ${\varepsilon=\gamma_k/2}$.

Let $a_0=0$ and $r_0=1$ and let $((a_n,r_n))_{n=1}^\infty$ be an enumeration (without repeats) of the set
\[
\bigcup_{\substack{m,n\in\N\\ n\geq 2}} \{(a_n^{(m)},r_n^{(m)}),(b_n^{(m)},s_n^{(m)})\}.
\]
Then $\mathbf A=((a_n,r_n))_{n=0}^\infty$ is an abstract Swiss cheese with $C_r\subseteq X_{\mathbf A}$. It is not hard to see that $\mathbf A$ is classical and $\rho(\mathbf A)<\varepsilon$. It is clear that $\tint_{\C}(X_{\mathbf A})=\emptyset$. Also, for each $n\in\N$, let
\[
E_n:=\partial\ob{a_0\diff n}{r_0\diff n}\cup\partial \ob{a_1\diff n}{r_1\diff n}\cup\partial\ob{b_0\diff n}{s_0\diff n}\cup\partial \ob{b_1\diff n}{s_1\diff n},
\]
(where $\partial S$ denotes the boundary of $S\subseteq\C$) and set $E:=C_r\cup \bigcup_{n=1}^\infty E_n$. Set $X:=X_{\mathbf A}$ and $U:=X\setminus E$. It is easy to see that $E$ is a closed set and that $U$ is a dense open subset of $X$. Moreover, $E$ has area zero .

We {\em claim} that each point $x\in U$ is a point of continuity for $R(X)$. Let $x\in U$. Then there exists a unique $n\in\N$ such that either $x\in X_{\mathbf A_n}$ or $x\in X_{\mathbf B_n}$. If $x\in X_{\mathbf A_n}$, set $Y=X_{\mathbf A_n}$, and if $x\in X_{\mathbf B_n}$ then set $Y=X_{\mathbf B_n}$. By our construction $R(Y)$ is regular and it is not hard to see that $x\in\tint_X{Y}$. Thus, by Lemma \ref{subset interior points of continuity}, $x$ is a point of continuity for $R(X)$. This proves the claim.

It remains to show that (c) holds. We first consider functions in $R_0(X)$.
Let $z\in C_r$ and let $f\in R_0(X)$. For each $n\in\N$, let $d_n$ the distance from $\cb{a_n}{r_n}$ to $C_r$, and let $d_0=1-r$. Since each $\mathbf A_n$ and each $\mathbf B_n$ are classical, and since $C_r\nsubseteq X_{\mathbf A_n},X_{\mathbf B_n}$ for all $n\in\N$, it follows that $d_n>0$ for all $n\in\No$. By Lemma \ref{R(X)_cheese_estimates}, for each $k\in\No$, we have
\[
\abs{f\diff k(z)}\leq k!\sum_{j=0}^\infty\frac{r_j}{d_j^{k+1}}\abs{f}_{X}
\]
Fix $m\in\N$. Then there exists a unique $n\in\N$ such that there exists $\ell\in\N$ with $(a_m,r_m)=(a_\ell\diff n,r_\ell\diff n)$ or there exists $\ell\in\N$ with $(a_m,r_m)=(b_\ell\diff n,s_\ell\diff n)$. In either case, since $\mathbf A_n$ and $\mathbf B_n$ are classical, we have $d_m>2^{3-n_0-m}$. Thus
\[
k!\sum_{j=0}^\infty\frac{r_j}{d_j^{k+1}}\abs{f}_{X}\leq k!\abs{f}_{X}\left(\frac1{d_0^{k+1}}+\sum_{j=1}^\infty\frac{\gamma_j}{(2^{1-n_0-j})^{k+1}}\right).
\]
for each $k\in\No$. It follows that, for each $k\in\No$, we have
\[
\abs{f\diff k}_{C_r}\leq k!\abs{f}_{X}\left(\frac1{d_0^{k+1}}+(\log{(k+3)})^{k}\right).
\]
From this we deduce that, for each $f\in R(X)$ (not necessarily in $R_0(X)$), we have $f|C_r\in\od{\infty}{C_r}$ and the same estimates hold. (One way to see this is to note that $C_r$ is uniformly regular, and apply \cite[Theorem~5.6]{dalefein2010}.)

Now let $f\in R(X)$. As in \cite{FY2015}, choose $N\in\N$ large enough so that $(\log{(k+3)})^k\geq 1/d_0^{k+1}$ for all $k\in\N$ with $k\geq N$. Then we have
\[
\sum_{j=1}^\infty\frac1{\abs{f\diff j}_{C_r}^{1/j}}\geq \sum_{j=N}^{\infty}\frac1{(2\abs{f}_X)^{1/j}j\log{(j+3)}}=\infty,
\]
and so (c) holds. This completes the proof.
\end{proof}

We are now ready to prove Theorem \ref{trivial Jensen cheese main theorem}.

\begin{proof}[Proof of Theorem~\textup{\ref{trivial Jensen cheese main theorem}}]
Apply Lemma \ref{construction} with $r=1/2$ and $\varepsilon=1$ to obtain a classical abstract Swiss cheese $\mathbf A$ which satisfies properties (a)--(c) described in the statement of the lemma, and let $X:=X_{\mathbf A}$. Then, for each $f\in R(X)$, we have $f|C_r\in\od\infty X$ and
\[
\sum_{n=1}^\infty\abs{f\diff n}_{C_r}^{-1/n}=\infty.
\]
So, by Corollary \ref{Useful Den Carl}, if $f\in R(X)$ for which $f^{(k)}(z)=0$ for some $z\in C_r$ and all $k\in\No$ then $f$ is identically zero on $C_r$. It follows that no point of $C_r$ can be a point of continuity for $R(X)$ and therefore $R(X)$ is not regular. It remains to see that $R(X)$ has no non-trivial Jensen measures on $X$. By Lemma~\ref{construction}(b) there is a dense open subset $U$ of $X$ for which every point $z\in U$ is a point of continuity for $R(X)$ and such that $X\setminus U$ has area $0$. Set $E:=X\setminus U$. By Lemma~\ref{punctured neighbourhood R point}, every point of $U$ is also an $R$-point. Let $x\in C_r$ and let $\mu$ be a Jensen measure for $\varepsilon_x$. Then since $x\notin U$, by Lemma \ref{non R points in support}, $\supp\mu\cap U=\emptyset$ and so $\supp\mu\subseteq E$. Since $\tint_{\C}(X)=\emptyset$ and the area of $E$ is $0$, it follows from Corollary~\ref{jensen measures supports corollary} that $\mu$ must be trivial. This completes the proof.
\end{proof}

It is also possible to show that $R(X_{\mathbf A})$ admits no non-trivial Jensen measures by appealing to the theory of Jensen interior. (See, for example, \cite[p.~319]{gamelin1983}.) This is the approach used in \cite{feinstein2001}.

Our final corollary follows immediately from Theorem~\ref{trivial Jensen cheese main theorem}.

\begin{corollary}
There exists a locally connected compact plane set $X$ such that $R(X)$ is essential$,$ non-trivial and non-regular and yet $R(X)$ admits no non-trivial Jensen measures.
\end{corollary}

We conclude with some open questions.

\begin{question}
Is the uniform algebra $R(X_{\mathbf A})$ constructed in Theorem~\ref{trivial Jensen cheese main theorem} necessarily antisymmetric? If not, can the construction be modified to yield an example which has the properties in that theorem and is also antisymmetric?
\end{question}

\begin{question}
Let $A$ be a uniform algebra on a compact Hausdorff space $X$, and let $M_i~(i\in I)$ be the decomposition of $X$ into maximal $A$-antisymmetric subsets.
\begin{enumerate}
\item
Suppose that $A|{M_i}$ is regular on $M_i$ for all $i \in I$. Must $A$ be regular on $X$? What if we assume the stronger condition that $A|M_i$ is regular (so natural and regular on $M_i$) for all $i \in I$?
\item
What is the answer to (a) in the special case where $X$ is a compact plane set and $A=R(X)$?
\end{enumerate}
\end{question}

\bibliographystyle{abbrvnat}

\end{document}